\documentclass[11pt]{amsart}

\usepackage[hmargin=0.8in,height=8.6in]{geometry}
\usepackage{amssymb,amsthm, times}
\usepackage{delarray,verbatim}
\usepackage{ifpdf}
\ifpdf
\usepackage[pdftex]{graphicx}
 \else
\usepackage[dvips]{graphicx}
 \fi

\usepackage{bm}

\usepackage{ifpdf}
\usepackage{color,soul}
\usepackage[bordercolor=white,backgroundcolor=gray!30,linecolor=black,colorinlistoftodos]{todonotes}

\usepackage{xcolor}
\usepackage{mdframed}
{\begin{mdframed}[backgroundcolor=yellow]\begin{remark}}%
		{\end{remark}\end{mdframed}}
\definecolor{webgreen}{rgb}{0,.5,0}
\definecolor{webbrown}{rgb}{.8,0,0}
\definecolor{emphcolor}{rgb}{0.95,0.95,0.95}
\usepackage{hyperref}
\hypersetup{%
          colorlinks=true,
          linkcolor=webbrown,
          filecolor=webbrown,
          citecolor=webgreen,
          breaklinks=true}
\ifpdf \hypersetup{pdftex,
            pdfstartview=FitH, 
            bookmarksopen=true,
            bookmarksnumbered=true
} \else \hypersetup{dvips} \fi

\usepackage{color}

\numberwithin{equation}{section}

\newtheorem{theorem}{Theorem}[section]
\newtheorem{proposition}{Proposition}[section]

\newtheorem{remark}{Remark}[section]
\newtheorem{lemma}{Lemma}[section]

\newtheorem{assump}{Assumption}[section]

\numberwithin{remark}{section} \numberwithin{proposition}{section}
\numberwithin{corollary}{section}
\newcommand {\R}{\mathbb{R}}

\newcommand {\p}{\mathbb{P}}

\newcommand {\E}{\mathbb{E}}

\newcommand{\diff}{{\rm d}}

\newcommand{\lev}{L\'{e}vy }

\newcommand{\e}{\mathbb{E}}

\title{American options under periodic exercise opportunities}
\begin{document}

\begin{abstract} In this paper, we study a version of the perpetual American call/put
option where exercise opportunities arrive only periodically. Focusing
on the exponential \lev models with i.i.d.\ exponentially-distributed exercise
intervals, we show the optimality of a barrier strategy that exercises
at the first exercise opportunity at which the asset price is above/below a
given barrier.  Explicit solutions are obtained for the cases where the
underlying \lev process has only one-sided jumps. \\
\noindent \small{\noindent  AMS 2010 Subject Classifications: 60G40, 60J75, 91G80 }
\end{abstract}

	\thanks{This version: \today.  J. L. P\'erez  is  supported  by  CONACYT,  project  no.\ 241195.
K. Yamazaki is in part supported by MEXT KAKENHI grant no.\  17K05377.}
\author[J. L. P\'erez]{Jos\'e-Luis P\'erez$^*$}
\thanks{$*$\, Department of Probability and Statistics, Centro de Investigaci\'on en Matem\'aticas A.C. Calle Jalisco s/n. C.P. 36240, Guanajuato, Mexico. Email: jluis.garmendia@cimat.mx.  }
\author[K. Yamazaki]{Kazutoshi Yamazaki$^\dag$}
\thanks{$\dag$\, Department of Mathematics,
Faculty of Engineering Science, Kansai University, 3-3-35 Yamate-cho, Suita-shi, Osaka 564-8680, Japan. Email: kyamazak@kansai-u.ac.jp.   }
\date{}

	\maketitle
\section{Introduction} 
The valuation of the perpetual American put/call options has been considered in many papers.  This can be used as an approximation to the finite maturity case, and also used in various real option models (see, e.g., Dixit and Pindyck \cite{DP}). Unlike the finite maturity case that commonly requires numerical approaches, the perpetual case typically admits a simple analytical solution.
In particular, when the underlying process is an exponential \lev process, the optimal stopping time is known to be of barrier-type: it is optimal to exercise as soon as the process goes above or below a certain barrier, which can be written concisely using the Wiener-Hopf factors (see Remark \ref{remark_classical_case}).
We refer the readers to the seminal papers by Mordecki \cite{M} and Alili and Kyprianou \cite{AK}, in this context.

In this paper, we consider a variant where exercise opportunities arrive only periodically.  While most of the continuous-time models assume that one can exercise the option instantaneously at any time, in reality one can monitor the underlying process only at intervals. Motivated by this, we consider the case in which one can exercise only at the jump times of an independent Poisson process.  This can be seen as a modification of Bermudan options where the constant exercise intervals are replaced by i.i.d.\ exponential random variables. 


We consider both put- and call-type payoffs.  The objective of this paper is twofold.

First, for a general underlying \lev process, we show the optimality of the \emph{periodic barrier strategy} that exercises at the first exercise opportunity at which the underlying process is below/above a suitably chosen barrier. 


Second, we focus on the spectrally one-sided \lev case (where jumps are one-sided) and obtain the optimal solution explicitly using the scale function. The expected value under a periodic barrier strategy can be directly computed using the results in Albrecher et al.\  \cite{AIZ}.  We obtain the optimal barriers and derive explicit forms of the optimal value function.  


In order to confirm the analytical results, we also give numerical results using spectrally one-sided \lev processes consisting of a Brownian motion with  i.i.d.\ exponential-size jumps.  We confirm 
the optimality, and also study the behavior of the optimal solutions with respect to the rate of Poisson arrivals. 


This paper is motivated by recent developments on  the optimal dividend problem where one wants to maximize the total discounted  dividends until ruin, with an extra restriction that the dividend payment opportunities arrive only periodically. 
It has recently been shown, for the case of exponential interarrival times, that a periodic barrier strategy is optimal when the underlying process is a spectrally one-sided \lev process (see \cite{ATW2014, NPYY2017, YP2016}). This current paper can be seen as its optimal stopping version. Other related research includes Parisian ruin/reflection (see, e.g., \cite{APY, PYM}); when the (Parisian) delays are exponential random variables, many fluctuation identities can be written semi-analytically in terms of the scale function, similarly to what we discuss in this paper for the case of a spectrally negative \lev process.

The rest of the paper is organized as follows.  Section \ref{section_preliminary} gives a mathematical formulation of the problem. Section \ref{section_optimality_barrier} shows, for a general \lev case, the optimality of a periodic barrier strategy.  We then obtain optimal solutions explicitly for the spectrally one-sided case in Section \ref{section_spectrally_one_sided}.  We conclude with numerical results in Section \ref{section_numerics}.

\section{Our Problem} \label{section_preliminary}


Let $X$ be a \lev process and  $S = \exp X$ the price of a stock.  For $s>0$, 
we denote by $\p_s$ the law of $S$ when it starts at $s$ ($X_0 = \log s$) and write for convenience  $\p$ in place of $\p_1$. Accordingly, we shall write $\e_s$ and $\e$ for the associated expectation operators. We define $\mathcal{T} := \{T_1, T_2, \ldots \}$ as the jump times of an independent Poisson process $N$ with rate $\lambda > 0$.  Let $\mathbb{F}$ be the filtration generated by  $(X,N)$ and $\mathbb{T}$ the set of  $\mathbb{F}$-stopping times.  The set of strategies is  given by $\mathcal{T}\cup \{\infty\}$-valued stopping times:
\begin{align*}
\mathcal{A} := \{ \tau \in \mathbb{T}: \tau \in \mathcal{T}\cup \{\infty\} \; a.s. \}.
\end{align*}


We consider American-type put/call options:
\begin{align}
V_i(s) = \sup_{\tau \in \mathcal{A}} \E_s [e^{-r \tau} G_i(S_\tau) 1_{\{ \tau < \infty \}}
], \quad i = p, c, \label{value_function} \end{align} 
for
\begin{align*}
G_p(s) := (K-s)^+ \quad \textrm{and} \quad G_c(s) := (s-K)^+,
\end{align*}
for a given discount factor $r > 0$ and strike price $K > 0$. 


In order to obtain a nontrivial solution, we assume the following for the call option.
\begin{assump} \label{assump_call}
For the call option $V_c$, we assume that $\E S_1 < e^{r}$. 
\end{assump}

\begin{remark} [Classical case]  \label{remark_classical_case}The classical case with the set of admissible strategies $\mathcal{A}$ replaced by $\mathbb{T}$ has been solved by  \cite{M, AK}. 
(i) In the classical put case, it is optimal to stop as soon as $S$ goes below the level 
\begin{align*}
A^*_{p, \infty}:= K \E [\exp (\underline{X}_{\mathbf{e}_{r}})],
\end{align*}
where $\underline{X}$ is the running infimum process of $X$ and $\mathbf{e}_{{r}}$ is an independent exponential random variable with parameter ${r}$.
(ii) In the classical call case, under Assumption \ref{assump_call}, it is optimal to stop as soon as $S$ goes above 
\begin{align*}
B^*_{c, \infty} := K \E [\exp (\overline{X}_{\mathbf{e}_{{r}}})],
\end{align*}
where $\overline{X}$ is the running supremum process of $X$.
\end{remark}

\section{Optimality of periodic barrier strategies} \label{section_optimality_barrier}


In this section, we show that the optimal stopping times  are of the form
\begin{align} \label{barrier_strategies}
\tau_{A}^- := \inf \{ T \in \mathcal{T}: S_T \leq A \} \quad \textrm{and} \quad \tau_{B}^+ := \inf \{ T \in \mathcal{T}: S_T \geq B \}
\end{align}
for suitably chosen barriers $A$ and $B$. Here and throughout, we let $\inf \varnothing = \infty$. 

Let $\bar{\mathcal{A}} := \{ \tau \in \mathbb{T}: \tau \in \bar{\mathcal{T}} \cup \{\infty\} \; a.s. \}$
with $\bar{\mathcal{T}} := \mathcal{T} \cup \{0\}$.
Define the value function of an auxiliary problem where immediate stopping is also allowed:
\begin{align} \label{extended_problem}
\bar{V}_i(s) = \sup_{\tau \in \bar{\mathcal{A}}} \E_s [e^{-{r} \tau} G_i(S_\tau) 1_{\{ \tau < \infty \}}], \quad i = p, c, \quad s > 0. \end{align}
By the strong Markov property,
\begin{align*}
V_i(s) = \E_s [e^{- {r} T_1} \bar{V}_i(S_{T_1})], \quad i = p, c, \quad s > 0.
\end{align*} 
If \eqref{extended_problem} is solved by 
\begin{align*} 
\bar{\tau}_{A}^- := \inf \{ T \in \bar{\mathcal{T}}: S_T \leq A \} \quad \textrm{and} \quad \bar{\tau}_{B}^+ := \inf \{ T \in \bar{\mathcal{T}}: S_T \geq B \},
\end{align*}
it is clear that \eqref{value_function}  is solved by \eqref{barrier_strategies} for the same values of $A$ and $B$. Hence, we shall analyze $\bar{V}_i$ below.

	Define the value function of an auxiliary finite-maturity problem:
\begin{align*}
\bar{V}_{i,n}^N(s) &:= \sup_{\tau \in \mathcal{A}_{n,N}} \E [e^{-{r} (\tau- T_n)} G_i(S_\tau) 
| S_{T_n} = s
], \quad 0 \leq n \leq N, \quad i = p,c, \quad s > 0,
\end{align*} 	
where 
\begin{align*}
\mathcal{A}_{n,N} &:= \{ \tau \in \bar{\mathcal{A}}: T_n \leq \tau \leq T_N  \, a.s. \}, \quad 0 \leq n \leq N,
\end{align*}
with $T_0 := 0$.
This is the expected value on condition that $S_{T_n} =s $ and the controller has not stopped before $T_n$.

Similarly, we define
\begin{align*}
\bar{V}_{i,n}(s) &:= \sup_{\tau \in \mathcal{A}_n} \E [e^{-{r} (\tau- T_n)} G_i(S_\tau) 
1_{\{\tau < \infty \}}
| S_{T_n} = s ],  \quad i = p,c, \quad s > 0,
\end{align*}
where 
\begin{align*}
	\mathcal{A}_n &:= \{ \tau \in \bar{\mathcal{A}}: \tau \geq T_n \, a.s. \}, \quad n \geq 0.
\end{align*}
It is clear that, for all $n \geq 0$ and $s > 0$,
\begin{align}
\bar{V}_{i,n}^N(s) = \bar{V}_{i,0}^{N-n}(s),  \; N \geq n, \quad \textrm{and} \quad \bar{V}_{i,n}(s) = \bar{V}_i(s). \label{v_n_same}
\end{align}
%
For the call case, by Assumption \ref{assump_call}, with $\alpha := \log \E S_1 < {r}$,
\begin{align*}
	0\leq \bar{V}_{c}(s) - \bar{V}_{c,0}^N (s) \leq e^s \sum_{n=N+1}^\infty \E [e^{-{r} T_n} S_{T_n}] =  e^s \sum_{n=N+1}^\infty \E [\E [e^{-{r} T_n} S_{T_n} |T_n] ] =  e^s \sum_{n=N+1}^\infty \E [ e^{-({r}-\alpha) T_n}]  \\= e^s \sum_{n=N+1}^\infty (\E [ e^{-({r}-\alpha) T_1}])^n =   e^s \sum_{n=N+1}^\infty \Big( \frac {\lambda} {\lambda+{r}-\alpha}\Big)^n.
\end{align*}
Hence, for $s > 0$, 
\begin{align}
\label{v_N_conv}
\begin{split}
&0 \leq \bar{V}_{p}(s) - \bar{V}_{p,0}^N(s) \leq \E_{s} [\sup_{t \geq T_N} e^{-{r}t} K] \xrightarrow{N \uparrow \infty} 0, \\
&0 \leq \bar{V}_{c}(s) - \bar{V}_{c,0}^N (s) \leq e^s \sum_{n=N+1}^\infty \Big( \frac {\lambda} {\lambda+{r}-\alpha}\Big)^n  \xrightarrow{N \uparrow \infty} 0.
\end{split}
\end{align}


By backward induction (similarly to the case of discrete-time stochastic dynamic programming), we can show the following.

\begin{proposition} \label{proposition_convexity}The mappings $s \mapsto \bar{V}_p(s)$ and $s \mapsto \bar{V}_c(s)$ are convex on $(0, \infty)$.
\end{proposition}
\begin{proof}
We focus on the put case; the proof for the call case is similar. First, we have
$\bar{V}_{p,N}^N(s) = (K- s)^+$,
which is convex.
By the dynamic programming principle (see, e.g., (1.1.53) of Peskir and Shiryaev \cite{PS}), we have
\begin{align}
\bar{V}_{p,n}^N(s) = \max \big((K- s)^+, \E_s [e^{-{r} T_1} \bar{V}_{p,n+1}^N(S_{T_1})] \big), \quad 0 \leq n \leq N-1. \label{dp_principle}
\end{align}
Suppose $\bar{V}_{p,n+1}^N$ is convex. Then,
\begin{align*}
\E_s [e^{-{r} T_1} \bar{V}_{p,n+1}^N(S_{T_1})] &= \lambda \int_0^\infty e^{-\lambda t} \E [e^{-{r} t} \bar{V}_{p,n+1}^N(s \exp(X_{t}))] \diff t = \lambda \int_0^\infty e^{-(\lambda +{r}) t} \int_{-\infty}^\infty \bar{V}_{p,n+1}^N(s e^{x}) \p (X_t \in \diff x ) \diff t ,
\end{align*}
is also convex in $s$.  Hence, $\bar{V}_{p,n}^N$ is also convex by \eqref{dp_principle}.
By induction, $\bar{V}_{p,0}^N$ is convex.  Finally, by the convergence \eqref{v_N_conv}, $\bar{V}_p$ is convex as well. 


\end{proof}

\begin{theorem} \label{theorem_barrier_optimal} There exist $K > A^*_p \geq A^*_{p, \infty} $ and $K < B^*_c \leq B^*_{c, \infty}$ such that $\tau_{A^*_p}^-$ and  $\tau_{B^*_c}^+$ solve \eqref{value_function}  for $i=p$ and $i=c$, respectively.
\end{theorem}
\begin{proof}
We only prove for the put case; the proof for the call case is similar. 
We consider the auxiliary problem  \eqref{extended_problem}. Immediate stopping gives $(K-s)^+$ and hence, we must have $\bar{V}_p(s) \geq (K-s)^+$.

Suppose $V_{p, \infty}$ is the value function in the classical case as in \cite{M, AK}. 
Then because $\mathcal{A} \subset \mathbb{T}$, we must have $\bar{V}_{p}(s) \leq V_{p, \infty} (s)$ for all $s > 0$.  In particular, because for $s < A^*_{p, \infty} < K$ (see Remark \ref{assump_call}), we have $V_{p, \infty}(s) = (K-s)^+ =K-s$ and hence we must have $\bar{V}_{p}(s) = (K-s)^+ = K-s$ as well.  By this and the convexity as in Proposition \ref{proposition_convexity}, we have $\{ s > 0: \bar{V}_{p}(s) = K-s\} = (0, A^*_p ]$ for some $A^*_p \in [A^*_{p, \infty}, \infty]$. In order to show that $A^*_p < K$ and $\{ s > 0: \bar{V}_{p}(s) = (K-s)^+ \} = (0, A^*_p ]$, it suffices to show $\bar{V}_p(s) > 0$ for all $s$.  Indeed, this holds because the payoff function is nonnegative and $S$ can reach any level below $K$ with a positive probability.  
Now, by the dynamic programming principle (see, e.g.\ Theorem 1.11 of Peskir and Shiryaev \cite{PS}), we have 
\begin{align*}
\bar{V}_{p}(s) = \max \big((K- s)^+, \E_s [e^{-{r} T_1} \bar{V}_{p}(S_{T_1})] \big), \quad s > 0,
\end{align*}
and $\inf \{ T \in \bar{\mathcal{T}} : \bar{V}_{p}(S_{T}) = (K-S_{T})^+ \} = \inf \{ T \in \bar{\mathcal{T}} : S_{T} \leq A^*_p \} = \bar{\tau}^-_{A^*_p}$ is the optimal strategy for the problem \eqref{extended_problem}.
Hence $\tau^-_{A^*_p}$ is optimal for \eqref{value_function}.


\end{proof}

\section{Spectrally one-sided case}  \label{section_spectrally_one_sided}


By Theorem \ref{theorem_barrier_optimal}, solving \eqref{value_function}  reduces to finding the best periodic barrier.  Here, we explicitly obtain it for the case of spectrally negative/positive \lev processes (that do not have monotone paths a.s.). We refer to Section 8 of Kyprianou \cite{K} for a comprehensive review.

Let $\tilde{\p}_x = \p_{\exp(x)}$ be the probability measure under which $S_0 = e^x$. Also define the stopping times
\begin{align} \label{crossing_time_X}
\begin{split}
\tilde{\tau}_{a}^- &:= \tau_{\exp(a)}^- = \inf \{ T \in \mathcal{T}: X_{T} \leq a \} \quad \textrm{and} \quad \tilde{\tau}_{b}^+ := \tau_{\exp(b)}^+ = \inf \{ T \in \mathcal{T}: X_{T} \geq b \}.
\end{split}
\end{align}
\subsection{Spectrally negative \lev case} Suppose $X$ is a spectrally negative \lev process with Laplace exponent:
\[
\psi(\theta) := \log \e\big[S_1^{\theta}\big] = \log \e\big[{\rm e}^{\theta X_1}\big], \qquad \theta\ge 0,
\]
with its right-inverse 
		\begin{align}
			\begin{split}
				\Phi(p) := \sup \{ s \geq 0: \psi(s) = p\}, \quad p > 0. 
			\end{split}
			\label{def_varphi}
		\end{align}

We use $W^{({r})}$ for the $r$-scale function of $X$.  This is the mapping from $\R$ to $[0, \infty)$ that takes value zero on the negative half-line, while on the positive half-line it is a continuous and strictly increasing function defined by
		\begin{align} \label{scale_function_laplace}
			\begin{split}
				\int_0^\infty  \mathrm{e}^{-\theta x} W^{({r})}(x) \diff x &= \frac 1 {\psi(\theta)-r}, \quad \theta > \Phi({r}).
			\end{split}
		\end{align}
	 The $(r+\lambda)$-scale function $W^{(r+\lambda)}$ is defined analogously.
Define also
\begin{align*} 
Z^{({r})}(x, \theta ) &:=e^{\theta x} \left( 1 + ({r}- \psi(\theta )) \int_0^{x} e^{-\theta  z} W^{({r})}(z) \diff z	\right), \quad x \in \R, \, \theta  \geq 0.
\end{align*}
In particular, for $x \in \R$, we let $Z^{({r})}(x) =Z^{({r})}(x, 0)$ and, for $\lambda > 0$,
\begin{align*} 
\begin{split}
Z^{({r})}(x, \Phi({r}+\lambda)) &=e^{\Phi(r+\lambda) x} \left( 1 -\lambda \int_0^{x} e^{-\Phi(r+\lambda) z} W^{({r})}(z) \diff z \right),  \\
Z^{({r}+\lambda)}(x, \Phi({r})) &=e^{\Phi({r}) x} \left( 1 + \lambda \int_0^{x} e^{-\Phi({r}) z} W^{(r+\lambda)}(z) \diff z	\right).
\end{split}
\end{align*}

By equation (14) of Theorem 3.1 in \cite{AIZ}, for $\theta \geq 0$,
\begin{align} \label{laplace_downcrossing}
\begin{split}
&\tilde{\mathbb{E}}_x \Big[e^{-{r} \tilde{\tau}_{a}^-+\theta X_{\tilde{\tau}_{a}^-}} 1_{\{ \tilde{\tau}_{a}^- < \infty \}}\Big] =e^{\theta a}\tilde{\mathbb{E}}_{x-a}\left[e^{-{r}\tilde{\tau}_{0}^-+\theta X_{\tilde{\tau}_{0}^-}} 1_{\{ \tilde{\tau}_{0}^- < \infty \}} \right]\\
&=\frac{\lambda e^{\theta a}}{\lambda +{r}-\psi(\theta)}\left(Z^{({r})}(x-a,\theta)-Z^{({r})}(x-a,\Phi({r}+\lambda))\frac{\psi(\theta)-{r}}{\lambda}\frac{\Phi({r}+\lambda)-\Phi({r})}{\theta-\Phi({r})}\right),
\end{split}
\end{align}
where the cases $\lambda+{r} = \psi(\theta)$ and ${r} = \psi(\theta)$ are interpreted as the limiting cases.
\begin{lemma}[Extension of equation (16) of \cite{AIZ}] \label{lemma_upcrossing_time} 
For $x, b \in \R$ and $\theta \geq 0$, we have 
\begin{align*}
\tilde{\E}_x \Big[ e^{-{r} \tilde{\tau}_b^+ -\theta (X_{\tilde{\tau}_b^+}-b)} 1_{\{\tilde{\tau}_b^+ < \infty \}}\Big] &= \frac {\Phi(\lambda+{r}) - \Phi({r})} {\Phi(\lambda+{r}) +\theta}  Z^{({r}+\lambda)} (x-b, \Phi({r})) - \lambda\int_0^{x-b} e^{-\theta y}  W^{({r}+\lambda)} (x-b-y)  \diff y. 
\end{align*}
\end{lemma}
\begin{proof}
Let $u(x)$ be the expectation on the left-hand side. By equation (16) of  \cite{AIZ}, for all $x \leq b$,
\begin{align} \label{u_below}
u(x) 
= \frac {\Phi(\lambda+{r}) - \Phi(r)} {\Phi(\lambda+{r}) +\theta} e^{-\Phi({r}) (b-x)}.
\end{align}
For $x > b$, by the strong Markov property and the memoryless property of the exponential random variable, with $\sigma_b^- := \inf \{ t > 0: X_t < b \}$, 
\begin{align}
u(x)
&= \tilde{\E}_x \Big[ e^{-{r}\sigma_b^-} u(X_{\sigma_b^-}) 1_{\{T_1 > \sigma_b^- \}} \Big]  + \tilde{\E}_x \Big[ e^{-{r} T_1-\theta (X_{T_1}-b)} 1_{\{T_1 < \sigma_b^- \}}\Big].  \label{u_decomposition}
\end{align}
Here, by \eqref{u_below} and identity (3.19) in \cite{APP},
\begin{align*}
\tilde{\E}_x \Big[ e^{-{r} \sigma_b^-} u(X_{\sigma_b^-}) 1_{\{T_1 > \sigma_b^-  \}}\Big] &= \frac {\Phi(\lambda+{r}) - \Phi({r})} {\Phi(\lambda+{r}) +\theta}  \tilde{\E}_x \Big[ e^{-{r} \sigma_b^-+\Phi({r}) (X_{\sigma_b^-}-b)} 1_{\{T_1 > \sigma_b^- \}}\Big]  \\
&= \frac {\Phi(\lambda+{r}) - \Phi({r})} {\Phi(\lambda+{r}) +\theta}  \tilde{\E}_x \Big[ e^{-({r}+\lambda) \sigma_b^- + \Phi({r}) (X_{\sigma_b^-}-b)} 1_{\{ \sigma_b^- < \infty \}} \Big] \\
&= \frac {\Phi(\lambda+{r}) - \Phi({r})} {\Phi(\lambda+{r}) +\theta}   \Big(Z^{({r}+\lambda)} (x-b, \Phi({r})) -  \frac {\lambda W^{({r}+\lambda)}(x-b)} { \Phi({r}+\lambda)-\Phi({r})} \Big).
\end{align*}
On the other hand, by Corollary 8.8 of \cite{K},
\begin{align*}
&\tilde{\E}_x \Big[ e^{-{r} T_1 -\theta (X_{T_1}-b)} 1_{\{T_1 < \sigma_b^- \}} \Big] = \tilde{\E}_x \Big[ \lambda \int_0^{\sigma_b^-} e^{-({r}+\lambda) t  -\theta (X_{t}-b)} \diff t  \Big] \\
&= \tilde{\E}_{x-b} \Big[ \lambda \int_0^{\sigma_0^-} e^{-({r}+\lambda) t -\theta X_{t}} \diff t  \Big] 
= \lambda \int_0^\infty e^{-\theta y} \Big( e^{-\Phi({r}+\lambda) y} W^{({r}+\lambda)} (x-b) - W^{({r}+\lambda)} (x-b-y) \Big)  \diff y \\
&= \frac {\lambda W^{({r}+\lambda)}(x-b)} {\Phi({r}+\lambda) + \theta}  - \lambda \int_0^{x-b}  e^{-\theta y} W^{({r}+\lambda)} (x-b-y)  \diff y.
\end{align*}
Substituting these in 
\eqref{u_decomposition} and after simplification, we have the claim.
\end{proof}

\subsubsection{Put case}
Define, for $x \in \R$ and $a < \log K$,
\begin{align*}
v_p^{SN}(x;a):=\mathbb{E}_{\exp(x)} \left[e^{-{r} \tau_{\exp(a)}^-}(K-S_{\tau_{\exp(a)}^-}) 1_{\{ \tau_{\exp(a)}^- < \infty \}}\right]=\tilde{\mathbb{E}}_x \left[e^{-{r} \tilde{\tau}_{a}^-}(K-e^{X_{\tilde{\tau}_{a}^-}}) 1_{\{ \tilde{\tau}_{a}^- < \infty \}} \right].
\end{align*}
Then, by Theorem \ref{theorem_barrier_optimal}, there exists $a_p^{SN} < \log K$ such that
\begin{align*}
V_p(s) =v_p^{SN} (\log s; a_p^{SN}) = \sup_{a \in (-\infty, \log K)} v_p^{SN} (\log s; a), \quad s \in \R.
\end{align*}

By \eqref{laplace_downcrossing}, we can write
\begin{align}\label{vf_2}
\begin{split}
v_p^{SN}(x;a)&=\frac{\lambda K}{\lambda+{r}}\left(Z^{({r})}(x-a)-Z^{({r})}(x-a,\Phi({r}+\lambda))\frac{{r}(\Phi({r}+\lambda)-\Phi({r}))}{\lambda \Phi({r})}\right) \\
&-\frac{\lambda e^{a}}{\lambda+{r}-\psi(1)}\left(Z^{({r})}(x-a,1)-Z^{({r})}(x-a,\Phi({r}+\lambda))\frac{\psi(1)-{r}}{\lambda}\frac{\Phi({r}+\lambda)-\Phi({r})}{1-\Phi({r})}\right),
\end{split}
\end{align}
where in particular, for $x \leq a$,
\begin{align}\label{vf_s<s*}
\begin{split}
v_p^{SN}(x;a)
&=\frac{\lambda K}{\lambda+{r}}\left(1-e^{\Phi({r}+\lambda)(x-a)}\frac{{r}(\Phi({r}+\lambda)-\Phi({r}))}{\lambda \Phi({r})}\right) \\
&-\frac{\lambda}{\lambda+{r}-\psi(1)}\left(e^{x}- e^{a (1-\Phi({r}+\lambda))} e^{\Phi({r}+\lambda)x}\frac{\psi(1)-{r}}{\lambda}\frac{\Phi({r}+\lambda)-\Phi({r})}{1-\Phi({r})}\right).
\end{split}
\end{align}
The maximizer $a_p^{SN}$ can be identified by the first order condition for \eqref{vf_s<s*}.
For $x < a$, we have
\begin{align}
\begin{split}
\frac \partial {\partial a} v_p^{SN}(x;a)
&= e^{\Phi({r}+\lambda) (x-a)} f(a; 1)
\end{split}
\end{align}
where
\begin{align}
f(a; \theta) &:= C_0 -C_1(\theta) e^{a}. \label{def_f}
\end{align}
Here, for all $\theta \in \R$ such that $\psi(\theta)$ exists,
\begin{align*}
C_0 &:=  \frac{K}{\lambda+{r}} \Phi({r}+\lambda) \frac{{r}(\Phi({r}+\lambda)-\Phi({r}))}{\Phi({r})} > 0, \\
C_1(\theta) &:=\frac{\Phi({r}+\lambda)-\theta}{\lambda+{r}-\psi(\theta)}  \frac{\psi(\theta)-{r}}{\theta-\Phi({r})} (\Phi({r}+\lambda)-\Phi({r})),
\end{align*}
where the cases $\lambda+{r} = \psi(\theta)$ and ${r} = \psi(\theta)$ are interpreted as the limiting cases.


By $\psi(0) = 0$ and the convexity of  $\psi$ on $[0, \infty)$, we have that $\psi(1) > s$ if and only if $1 > \Phi(s)$ for all $s > 0$, and therefore 
\begin{align}
\frac {\psi(1)-s} {1-\Phi(s)} > 0, \quad s > 0. \label{relation_psi_Phi}
\end{align}
Hence, $C_1(1) > 0$.
This means,
$f'(a; 1) = - C_1(1) e^a < 0$ and therefore $a \mapsto f(a; 1)$ is continuous and  strictly decreasing.  Finally,
$\lim_{a \downarrow -\infty} f(a; 1) = C_0
> 0$ and $\lim_{a \uparrow \infty} f(a; 1) = -\infty$.
Hence, there exists a unique root of $f(\cdot ; 1)=0$ and it becomes the optimal barrier $a_p^{SN}$.


Using that $f(a_p^{SN};1)=0$ in \eqref{vf_2}, simple algebra gives
\begin{align*}
v_p^{SN}(x;a_p^{SN})
&=\frac{\lambda K}{\lambda+{r}} Z^{({r})}(x-a_p^{SN}) -\frac{\lambda e^{a_p^{SN}}}{\lambda+{r}-\psi(1)} Z^{({r})}(x-a_p^{SN},1) \\
&+ \frac {Z^{({r})}(x-a_p^{SN},\Phi({r}+\lambda))} {\Phi({r}+\lambda)} \frac{\psi(1)-{r}}{\lambda+{r}-\psi(1)} \frac{\Phi({r}+\lambda)-\Phi({r})}{1-\Phi({r})} e^{a_p^{SN} }. 
\end{align*}

\subsubsection{Call case}
First, Assumption \ref{assump_call} is equivalent to $\psi(1) < {r}$, and hence (by
\eqref{relation_psi_Phi}),
\begin{align} \label{Phi_order}
\Phi({r}+\lambda) > \Phi({r})  > 1.
\end{align}

By Lemma \ref{lemma_upcrossing_time}, extended to the case $\theta =-1$ by analytic continuation by \eqref{Phi_order}, we have
\begin{align} \label{v_c_SN_x_general}
\begin{split}
v_{c}^{SN}(x; b) &:=\mathbb{E}_{\exp(x)} \left[e^{-{r} \tau_{\exp(b)}^+}(S_{\tau_{\exp(b)}^+}-K) 1_{\{ \tau_{\exp(b)}^+ < \infty \}}\right]=\tilde{\mathbb{E}}_x \left[e^{-{r} \tilde{\tau}_{b}^+}(\exp(X_{\tilde{\tau}_{b}^+})-K) 1_{\{ \tilde{\tau}_{b}^+ < \infty \}} \right] \\
&= e^b \Big(\frac {\Phi(\lambda+{r}) - \Phi({r})} {\Phi(\lambda+{r}) -1}   Z^{({r}+\lambda)} (x-b, \Phi({r})) - \lambda \int_0^{x-b} e^{ y} W^{({r}+\lambda)} (x-b-y)   \diff y \Big) \\
&- K \Big(\frac {\Phi(\lambda+{r}) - \Phi({r})} {\Phi(\lambda+{r})}  Z^{({r}+\lambda)} (x-b, \Phi({r})) - \lambda \overline{W}^{({r}+\lambda)} (x-b)
  \Big),
  \end{split}
\end{align}
where $\overline{W}^{({r}+\lambda)}(z) := \int_0^z W^{({r}+\lambda)}(y) \diff y$ for $z \in \R$.
In particular, for $x\leq b$,
\begin{align} \label{v_c_SN_x_below_b}
v_{c}^{SN}(x; b) 
&= e^{\Phi({r}) x} \Big[   \frac {\Phi(\lambda+{r}) - \Phi({r})} {\Phi(\lambda+{r}) -1} e^{-(\Phi({r})-1) b}  -K \frac {\Phi(\lambda+{r}) - \Phi({r})} {\Phi(\lambda+{r})} e^{-\Phi({r}) b}\Big].
\end{align}
Similarly to the put case, we obtain the maximizer $b_c^{SN}$ by the first order condition for \eqref{v_c_SN_x_below_b}. We have
\begin{align*}
\frac \partial {\partial b} v_{c}^{SN}(x; b) = e^{\Phi({r}) (x-b)} g(b; -1), \quad b \in \mathbb{R}, \quad x < b,
\end{align*}
where, for $\theta \neq -\Phi(\lambda+{r})$,
\begin{align}
g(b; \theta) := K \Phi(r)  \frac {\Phi(\lambda+{r}) - \Phi({r})} {\Phi(\lambda+{r})}  - (\Phi({r}) + \theta) \frac {\Phi(\lambda+{r}) - \Phi({r})} {\Phi(\lambda+{r}) + \theta} e^{b}.  \label{def_g}
\end{align}
By \eqref{Phi_order}, $b \mapsto g(b; -1)$ is continuous and monotonically strictly decreasing with $\lim_{b \downarrow -\infty} g(b; -1) > 0$ and $\lim_{b \uparrow \infty} g(b; -1) = -\infty$. Hence, its unique root becomes the optimal barrier $b^{SN}_c$.

Applying that $g(b_c^{SN};-1)=0$ in \eqref{v_c_SN_x_general},  we have 
	\begin{align*}
		v_{c}^{SN}(x; b^{SN}_c) &=e^{b^{SN}_c}  \Big( \frac 1 {\Phi({r}) } \frac {\Phi(\lambda+{r}) - \Phi({r})} {\Phi(\lambda+{r}) - 1}  Z^{({r}+\lambda)} (x-b^{SN}_c, \Phi({r})) - \lambda \int_0^{x-b^{SN}_c} e^{ y} W^{({r}+\lambda)} (x-b^{SN}_c-y)   \diff y \Big) \\
		&+ K  \lambda \overline{W}^{({r}+\lambda)} (x-b^{SN}_c).
	\end{align*}
\subsection{Spectrally positive \lev case}
We now suppose that $X$ is a spectrally positive  \lev process. 
Then, its dual $Y = -X$ becomes a spectrally negative \lev process.  Let $\hat{\p}_y$ be the probability where $Y_0 = y$ and  $\hat{\tau}$ be the crossing times  for $Y$ defined analogously to \eqref{crossing_time_X}. 
Let $\psi$, $\Phi$, and $W$ denote the Laplace exponent, its inverse, and the scale function, respectively, for the process $Y$.

\subsubsection{Put case}
By Lemma \ref{lemma_upcrossing_time}, for $a < \log(K)$ and $x \in \R$,
\begin{align} \label{v_p_SP}
\begin{split}
v_p^{SP}(x;a) 
&:= \mathbb{E}_{\exp(x)} \left[e^{-{r} \tau_{\exp(a)}^-}(K-S_{\tau_{\exp(a)}^-}) 1_{\{ \tau_{\exp(a)}^- < \infty \}}\right]
=\hat{\mathbb{E}}_{-x} \left[e^{-{r}\hat{\tau}_{-a}^+}(K-\exp (- Y_{\hat{\tau}_{-a}^+})) 1_{\{ \hat{\tau}_{-a}^+ < \infty \}} \right] \\
&=K \Big(\frac {\Phi(\lambda+{r}) - \Phi({r})} {\Phi(\lambda+{r})}  Z^{({r}+\lambda)} (a-x, \Phi({r})) - \lambda \overline{W}^{({r}+\lambda)} (a-x)  \Big) \\
&-  e^a \Big(\frac {\Phi(\lambda+{r}) - \Phi({r})} {\Phi(\lambda+{r}) +1}  Z^{({r}+\lambda)} (a-x, \Phi({r})) - \lambda \int_0^{a-x} e^{-y} W^{({r}+\lambda)} (a-x-y)   \diff y \Big). 
\end{split}
\end{align}
In particular, for $x \geq a$,
\begin{align*}
v_{p}^{SP}(x; a) = e^{-\Phi({r}) x} \Big[ K \frac {\Phi(\lambda+{r}) - \Phi({r})} {\Phi(\lambda+{r})} e^{\Phi({r}) a} - \frac {\Phi(\lambda+{r}) - \Phi({r})} {\Phi(\lambda+{r}) +1} e^{(\Phi({r})+1) a}  \Big].
\end{align*}
Differentiating this, we have (using $g$ defined in \eqref{def_g})
\begin{align*}
\frac \partial {\partial a} v_{p}^{SP}(x;a) = e^{-\Phi({r}) (x-a)} g(a;1), \quad x > a.
\end{align*}
Because $\Phi(\lambda+{r}) > \Phi({r}) > 0$,  the function $g(\cdot; 1)$ is continuous and strictly decreasing with $\lim_{a \downarrow -\infty} g(a; 1)> 0$ and $\lim_{a \uparrow \infty}g(a; 1)= -\infty$. Hence, its unique root becomes the optimal barrier $a_p^{SP}$. 

	Applying that $g(a_p^{SP};1)=0$ in \eqref{v_p_SP}, we have 
\begin{align*}
	v_{p}^{SP}(x; a_p^{SP})
	&=e^{a_p^{SP}} \Big( \frac 1 {\Phi({r}) } \frac {\Phi(\lambda+{r}) - \Phi({r})} {\Phi(\lambda+{r}) +1}  Z^{({r}+\lambda)} (a_p^{SP}-x, \Phi({r})) + \lambda\int_0^{a_p^{SP}-x} e^{-y}  W^{({r}+\lambda)} (a_p^{SP}-x-y)  \diff y   \Big)\\
	&- K \lambda \overline{W}^{({r}+\lambda)} (a_p^{SP}-x).
\end{align*}	
\subsubsection{Call case}
Under Assumption \ref{assump_call}, the domain of the Laplace exponent $\psi$ can be extended to $[-1, \infty)$ and
\begin{align}
\log \E S_1 = \log \E [e^{-Y_1}] = \psi(-1) < {r}. \label{cond_psi_SP}
\end{align}
For $b > \log (K)$ and $x \in \R$, define
\begin{align*}
v_c^{SP}(x; b) :=
\mathbb{E}_{\exp(x)} \left[e^{-{r} \tau_{\exp(b)}^+}(S_{\tau_{\exp(b)}^+}-K) 1_{\{ \tau_{\exp(b)}^+ < \infty \}}\right]
=\hat{\mathbb{E}}_{-x} \left[e^{-{r} \hat{\tau}_{-b}^-}(\exp (- Y_{\hat{\tau}_{-b}^-}) -K) 1_{\{ \hat{\tau}_{-b}^- < \infty \}} \right].
\end{align*}
%
%
	By \eqref{laplace_downcrossing} (with $X$ replaced by $Y$) which holds for $\theta = -1$ by analytic continuation using \eqref{cond_psi_SP}, we can write
\begin{align} \label{v_c_SP_def}
\begin{split}
v_c^{SP}(x;b)&= \frac{\lambda e^{b}}{\lambda+{r}-\psi(-1)}\left(Z^{({r})}(b-x,-1)-Z^{({r})}(b-x,\Phi({r}+\lambda))\frac{\psi(-1)-{r}}{\lambda}\frac{\Phi({r}+\lambda)-\Phi({r})}{-1-\Phi({r})}\right) \\
&-\frac{\lambda K}{\lambda+{r}}\left(Z^{({r})}(b-x)-Z^{({r})}(b-x,\Phi({r}+\lambda))\frac{{r}(\Phi({r}+\lambda)-\Phi({r}))}{ \lambda \Phi({r})}\right),
\end{split}
\end{align}
where in particular, for $x \geq b$,
\begin{align*}
\begin{split}
v_c^{SP}(x; b)
&=\frac{\lambda}{\lambda+{r}-\psi(-1)}\left(e^{x}- e^{b (1+\Phi({r}+\lambda))} e^{-\Phi({r}+\lambda)x}\frac{\psi(-1)-{r}}{\lambda}\frac{\Phi({r}+\lambda)-\Phi({r})}{-1-\Phi({r})}\right) \\
&-\frac{\lambda K}{\lambda+{r}}\left(1-e^{\Phi({r}+\lambda)(b-x)}\frac{{r}(\Phi({r}+\lambda)-\Phi({r}))}{\lambda\Phi({r})}\right).
\end{split}
\end{align*}
Differentiating this and using $f$ as in \eqref{def_f},
\begin{align*}
\begin{split}
\frac \partial {\partial b} v_c^{SP}(x;b)
&= e^{\Phi({r}+\lambda) (b-x)} f(b;-1), \quad x > b,
\end{split}
\end{align*}

Here, 
\begin{align*}
C_1(-1) &=\frac{\Phi({r}+\lambda)+1}{\lambda+{r}-\psi(-1)}  \frac{{r}-\psi(-1)}{1+\Phi({r})} (\Phi({r}+\lambda)-\Phi({r})),
\end{align*}
which is strictly positive by \eqref{cond_psi_SP}. Hence, as in the case of $f(\cdot; 1)$, we can obtain $b^{SP}_c$ as the unique root of $f(\cdot; -1) = 0$.
	Applying that $f(b_c^{SP};-1)=0$ in \eqref{v_c_SP_def},  we have 
\begin{align*}
v_c^{SP}(x;b_c^{SP})
&=-\frac{\lambda K}{\lambda+{r}} Z^{({r})}(b_c^{SP}-x) +\frac{\lambda e^{b_c^{SP}}}{\lambda+{r}-\psi(-1)} Z^{({r})}(b_c^{SP}-x,-1) \\
&+\frac {Z^{({r})}(b_c^{SP}-x,\Phi({r}+\lambda))} {\Phi({r}+\lambda)} \frac{\psi(-1)-{r}}{\lambda+{r}-\psi(-1)}  \frac{\Phi({r}+\lambda)-\Phi({r})}{-1-\Phi({r})}e^{b_c^{SP}}. \end{align*}

\section{Numerical results} \label{section_numerics}

We conclude the paper with numerical examples, using spectrally negative and positive \lev processes consisting of a Brownian motion and  i.i.d.\ exponential-size jumps.    These are special cases of the double exponential jump diffusion of  Kou \cite{Kou}. The scale functions $W^{(r)}$ and $W^{(r+\lambda)}$ can be obtained explicitly as in \cite{Egami_Yamazaki_2010_2, KKR}.

For the spectrally negative case, we assume
\begin{equation}
 X_t - X_0= c t+ 0.2 B_t - \sum_{n=1}^{N_t} Z_n, \quad 0\le t <\infty, \label{X_phase_type}
\end{equation}
where $B=( B_t; t\ge 0)$ is a standard Brownian motion, $N=(N_t; t\ge 0 )$ is a Poisson process with arrival rate $1$, and  $Z = ( Z_n; n = 1,2,\ldots )$ is an i.i.d.\ sequence of exponential random variables with parameter $2$. The processes $B$, $N$, and $Z$ are assumed mutually independent.  For the spectrally positive case, we set $X$ to be the negative of \eqref{X_phase_type}.  The value $c$ is chosen so that $e^{-({r}-\delta)} S_t$ is a martingale for $\delta = 0.03$.  By this, Assumption \ref{assump_call} is satisfied. For other parameters, we set $K=50$, ${r} = 0.05$, and $\lambda=1$, unless stated otherwise.

Figures \ref{figure_put} and \ref{figure_call} show, respectively, the value functions $V_p(s)$ and $V_c(s)$ in comparison to the expected values under other barrier strategies. It can be confirmed that the obtained value functions dominate suboptimal value functions uniformly in the starting value $s$.  Also, it is observed that our choice of the optimal barrier is such that the value function becomes convex and also smooth at the barrier. 


In Figures \ref{figure_put_r} and \ref{figure_call_r}, we show $V_p(s)$ and $V_c(s)$ for various choices of the Poisson arrival rate $\lambda$ along with the classical case as in Remark \ref{remark_classical_case}. The value functions are increasing in $\lambda$ uniformly in $s$ and converge to the classical ones as $\lambda \uparrow \infty$.  On the other hand, as $\lambda \downarrow 0$, they decrease to zero. The optimal barriers converge to the ones in the classical case as $\lambda \uparrow \infty$ and to the strike price $K$ as $\lambda \downarrow 0$.

\begin{figure}[htbp]
\begin{center}
\begin{minipage}{1.0\textwidth}
\centering
\begin{tabular}{cc}
 \includegraphics[scale=0.32]{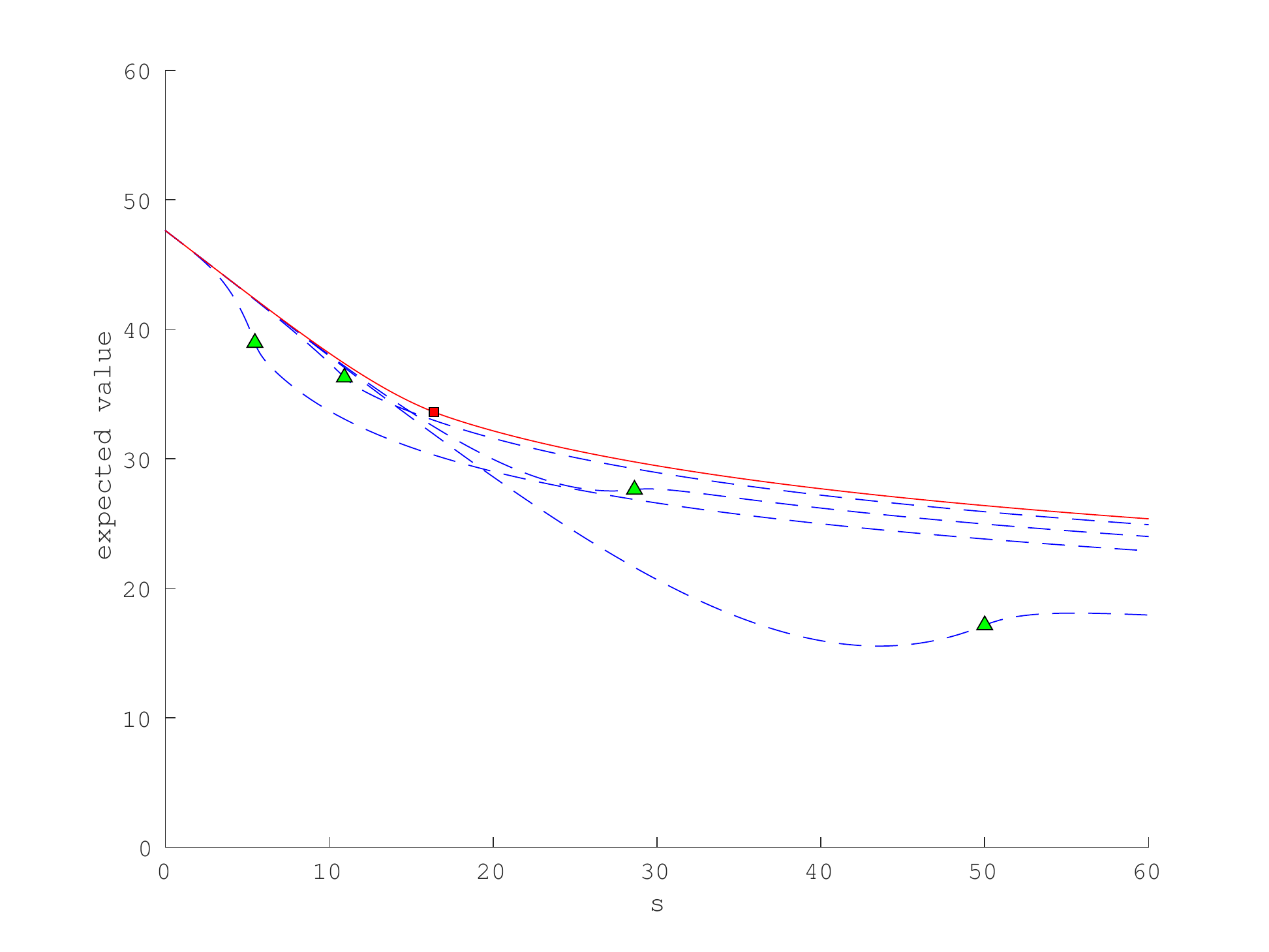} & \includegraphics[scale=0.32]{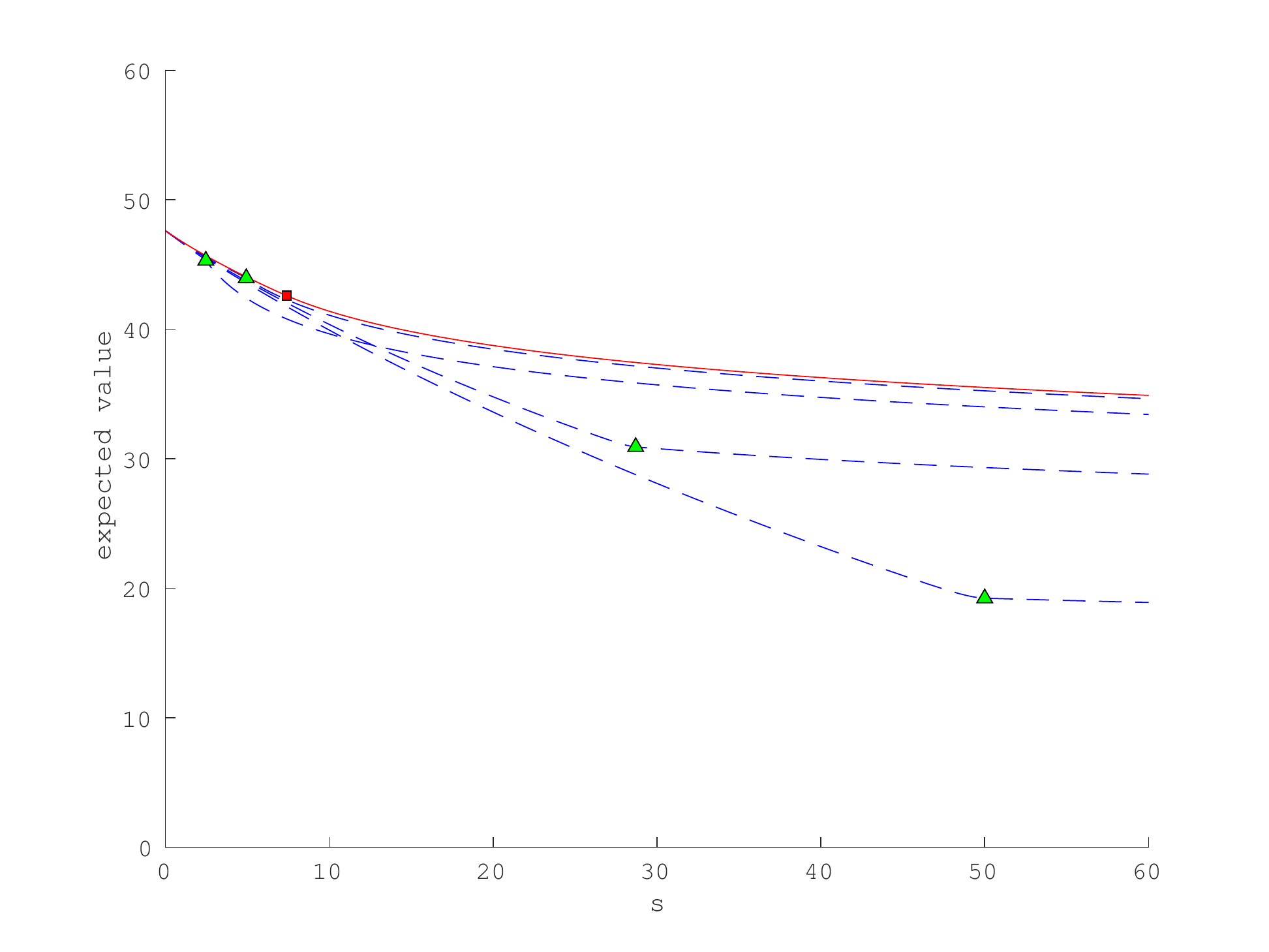}  \\
 spectrally negative case & spectrally positive case
  \end{tabular}
\end{minipage}
\caption{\footnotesize The results for the put case. (Left) $V_p^{SN}(s) = v_p^{SN}(\log s; a_p^{SN})$ along with $v_p^{SN}(\log s; a)$ for $\exp(a) = \exp(a_p^{SN})/3, 2 \exp(a_p^{SN})/3, (\exp(a_p^{SN}) + K)/2, K$. (Right) $V_p^{SP}(s) = v_p^{SP}(\log s; a_p^{SP})$ along with $v_p^{SP}(\log s; a)$ for $\exp(a) = \exp(a_p^{SP})/3, 2 \exp(a_p^{SP})/3, (\exp(a_p^{SP}) + K)/2, K$.
The values at $a_p^{SN}$ and $a_p^{SP}$ are indicated by  squares. Those at the suboptimal barriers $a$  are indicated by triangles.} 
\label{figure_put}
\end{center}
\end{figure}
\begin{figure}[htbp]
\begin{center}
\begin{minipage}{1.0\textwidth}
\centering
\begin{tabular}{cc}
 \includegraphics[scale=0.32]{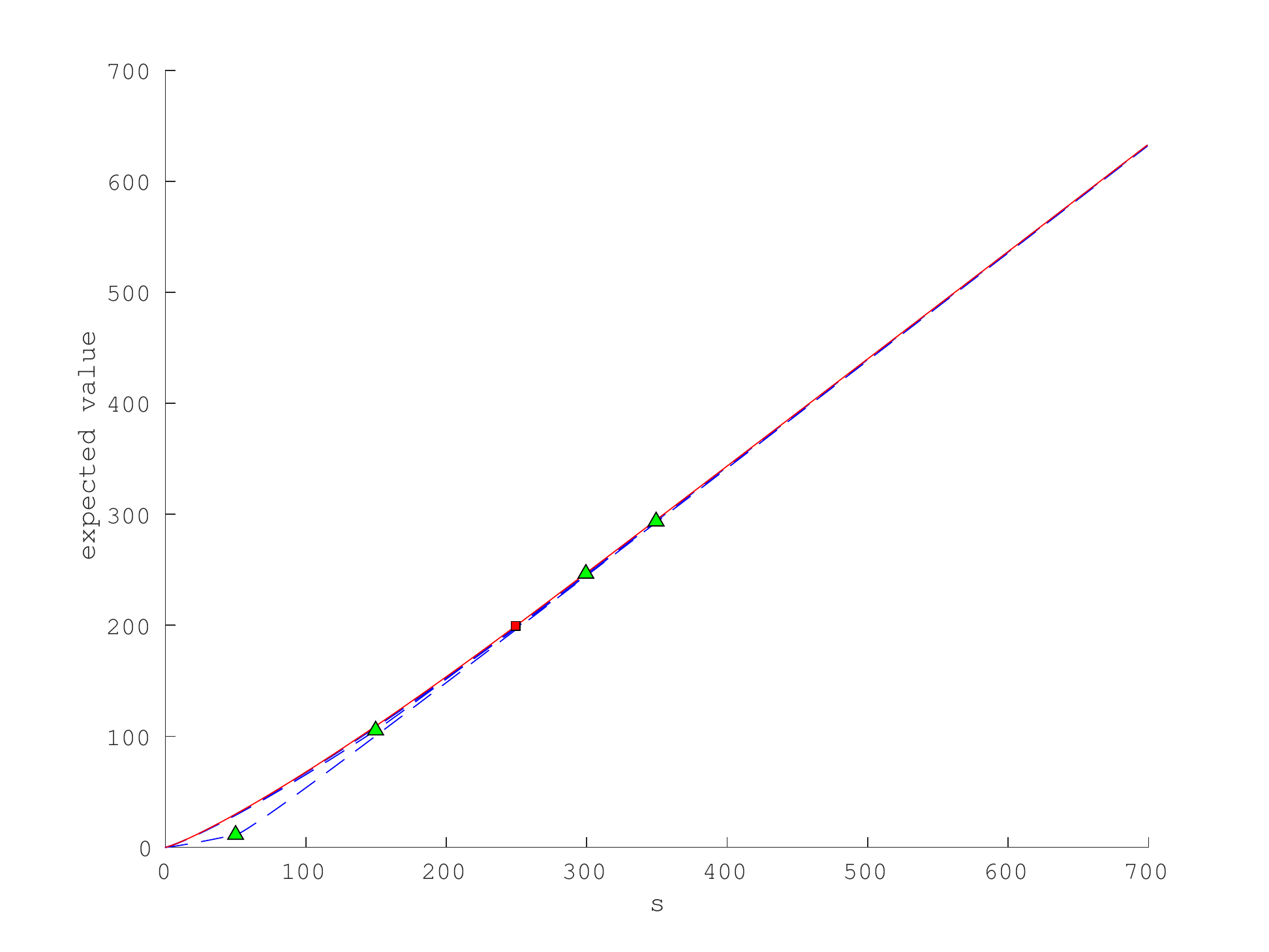} & \includegraphics[scale=0.32]{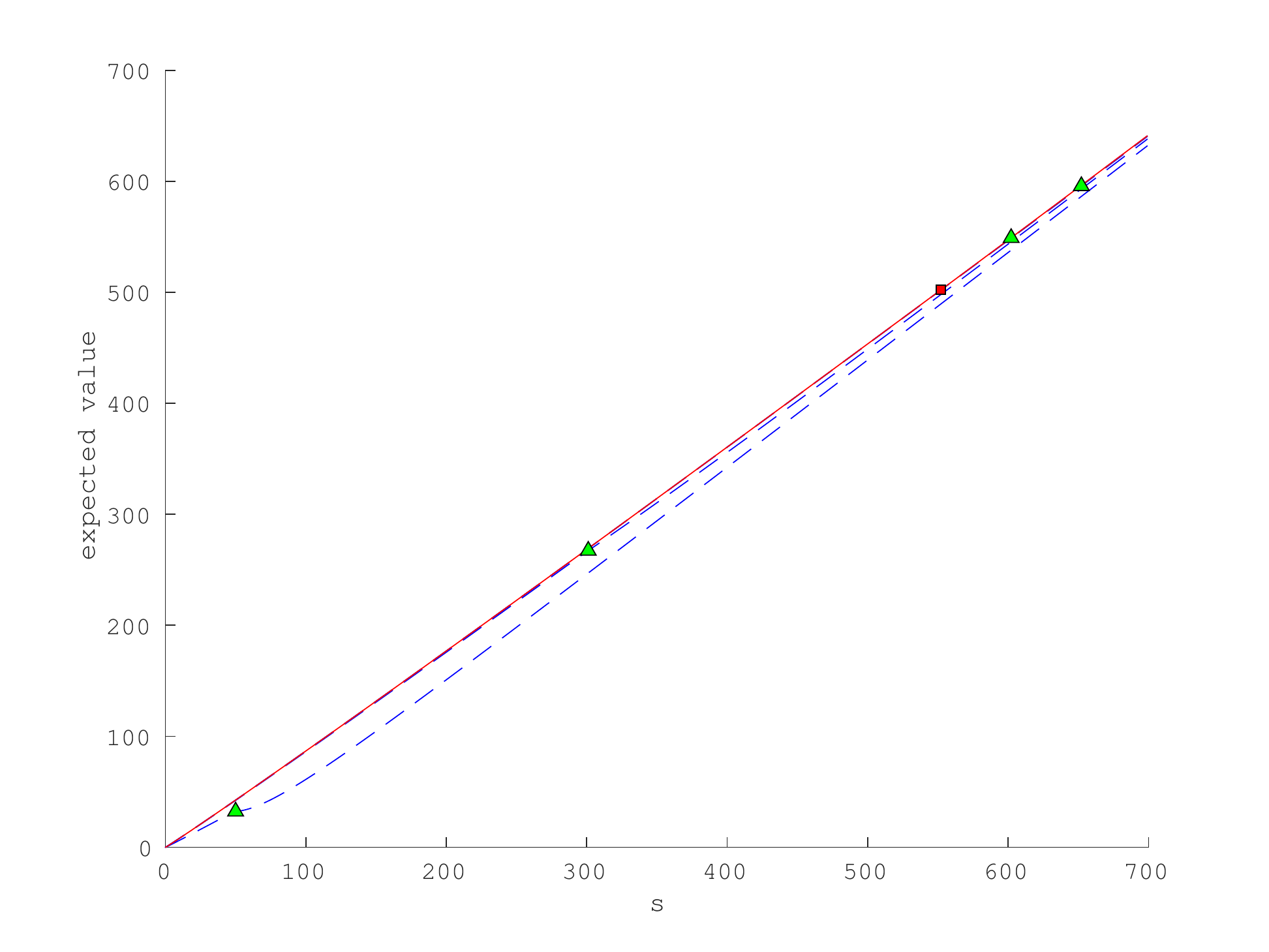}  \\
  spectrally negative case & spectrally positive case
 \end{tabular}
\end{minipage}
\caption{\footnotesize The results for the call case. (Left) $V_c^{SN}(s) = v_c^{SN}(\log s; b_c^{SN})$ along with $v_c^{SN}(\log s; b)$ for $\exp(b) = K, (\exp(b_c^{SN})+K)/2, \exp(b_c^{SN})+50, \exp(b_c^{SN})+100$. (Right) $V_c^{SP}(s) = v_c^{SP}(\log s; b_c^{SP})$ along with $v_c^{SP}(\log s, b)$ for $\exp(b) = K, (\exp(b_c^{SP})+K)/2, \exp(b_c^{SP})+50, \exp(b_c^{SP})+100$. 
The values at $b_c^{SN}$ and $b_c^{SP}$ are indicated by squares. Those at the suboptimal barriers $b$  are indicated by triangles.} 
\label{figure_call}
\end{center}
\end{figure}

\begin{figure}[htbp]
\begin{center}
\begin{minipage}{1.0\textwidth}
\centering
\begin{tabular}{ccc}
 \includegraphics[scale=0.32]{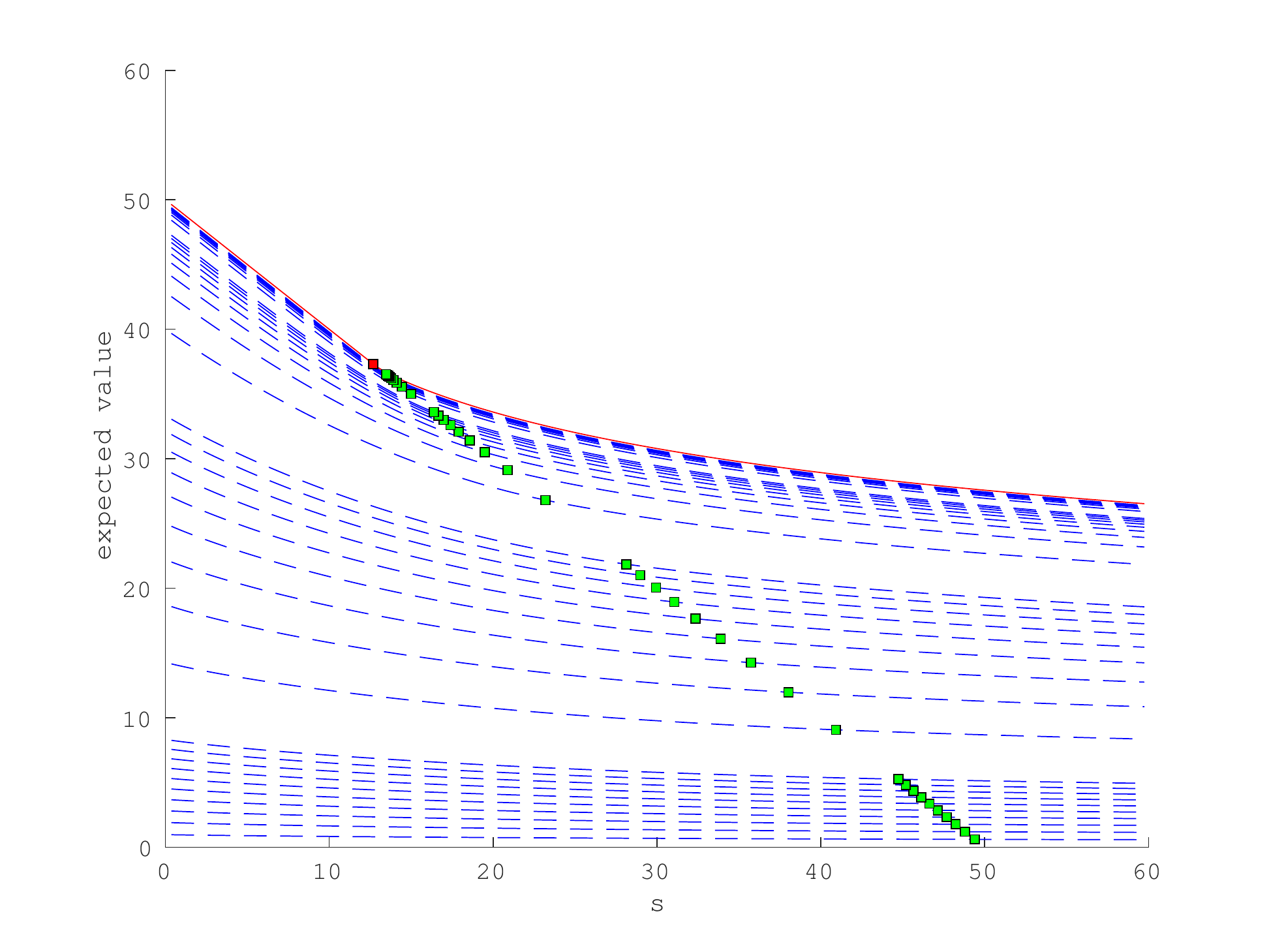} & \includegraphics[scale=0.32]{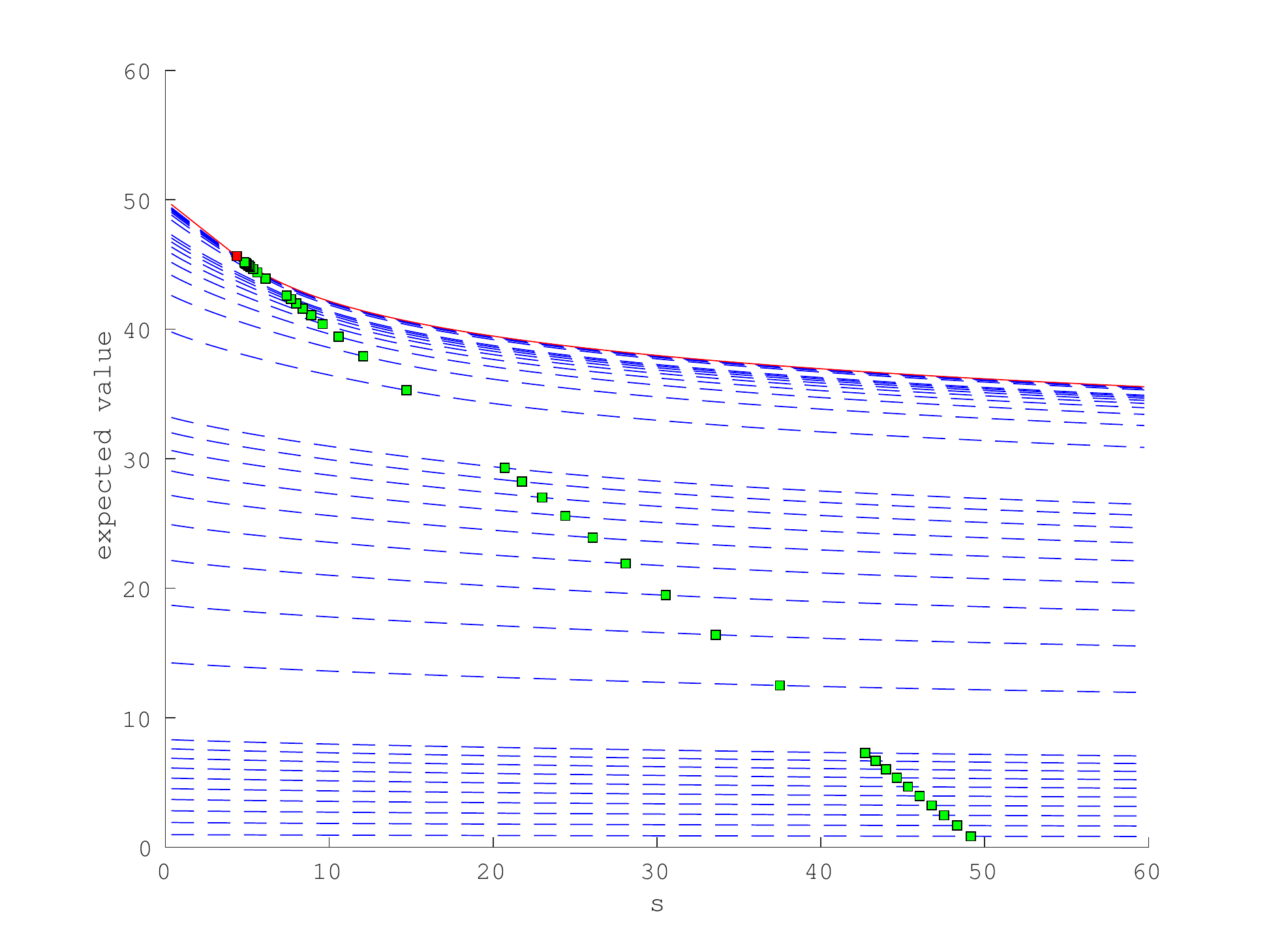}  \\
 spectrally negative case & spectrally positive case
 \end{tabular}
\end{minipage}
\caption{\footnotesize Plots of $V_p$ (dotted) for $\lambda = 0.001$, $0.002$, $\ldots$, $0.009$, $0.01$, $0.02$, $\ldots$, $0.09$, $0.1$,$0.2$, $\ldots$, $0.9$, $1$, $2$, $\ldots$, $10$ along with $V_{p, \infty}$ (solid).  The points at the optimal barriers are indicated by squares. 
} \label{figure_put_r}
\end{center}
\end{figure}
\begin{figure}[htbp]
\begin{center}
\begin{minipage}{1.0\textwidth}
\centering
\begin{tabular}{ccc}
 \includegraphics[scale=0.32]{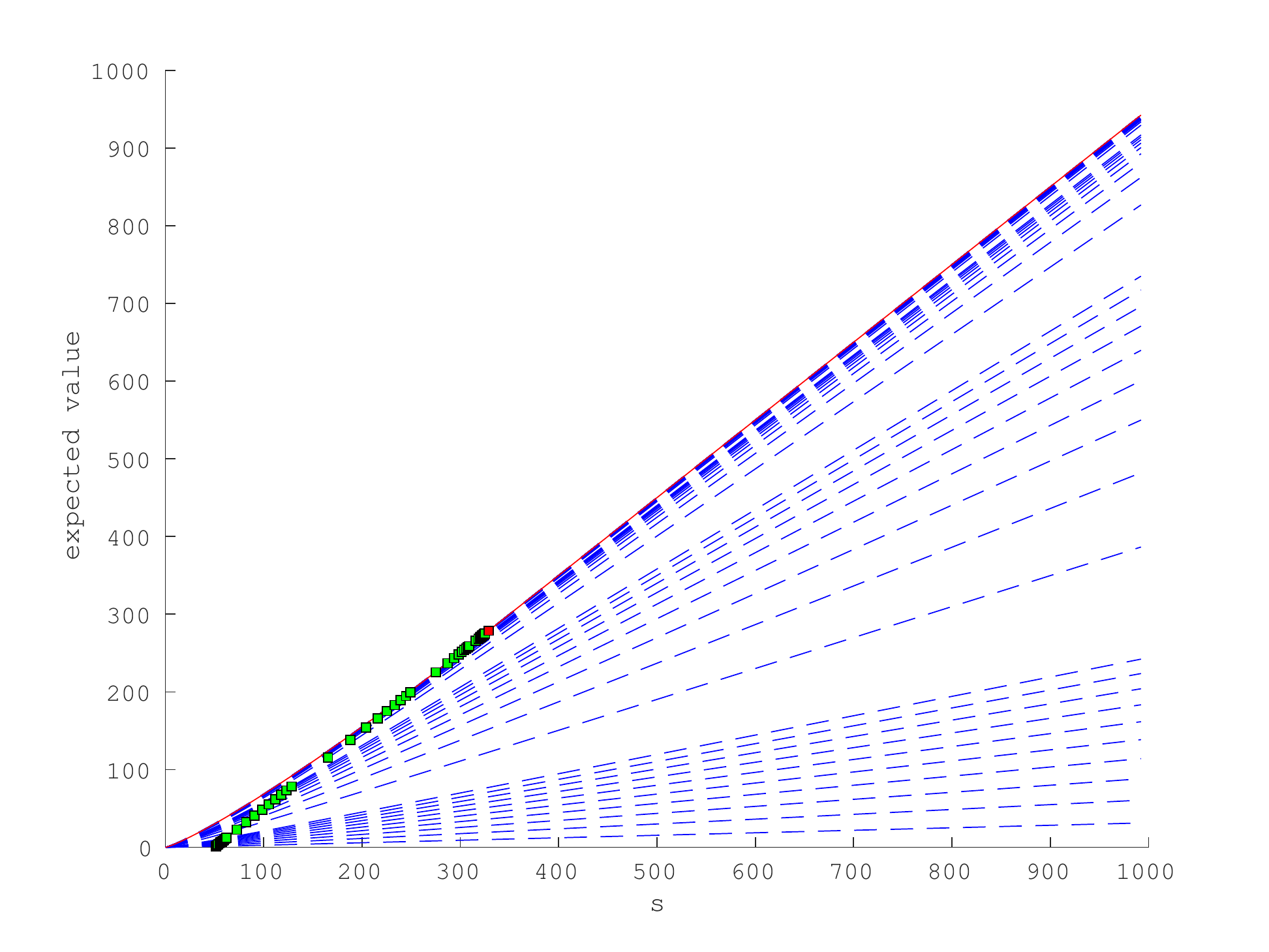} & \includegraphics[scale=0.32]{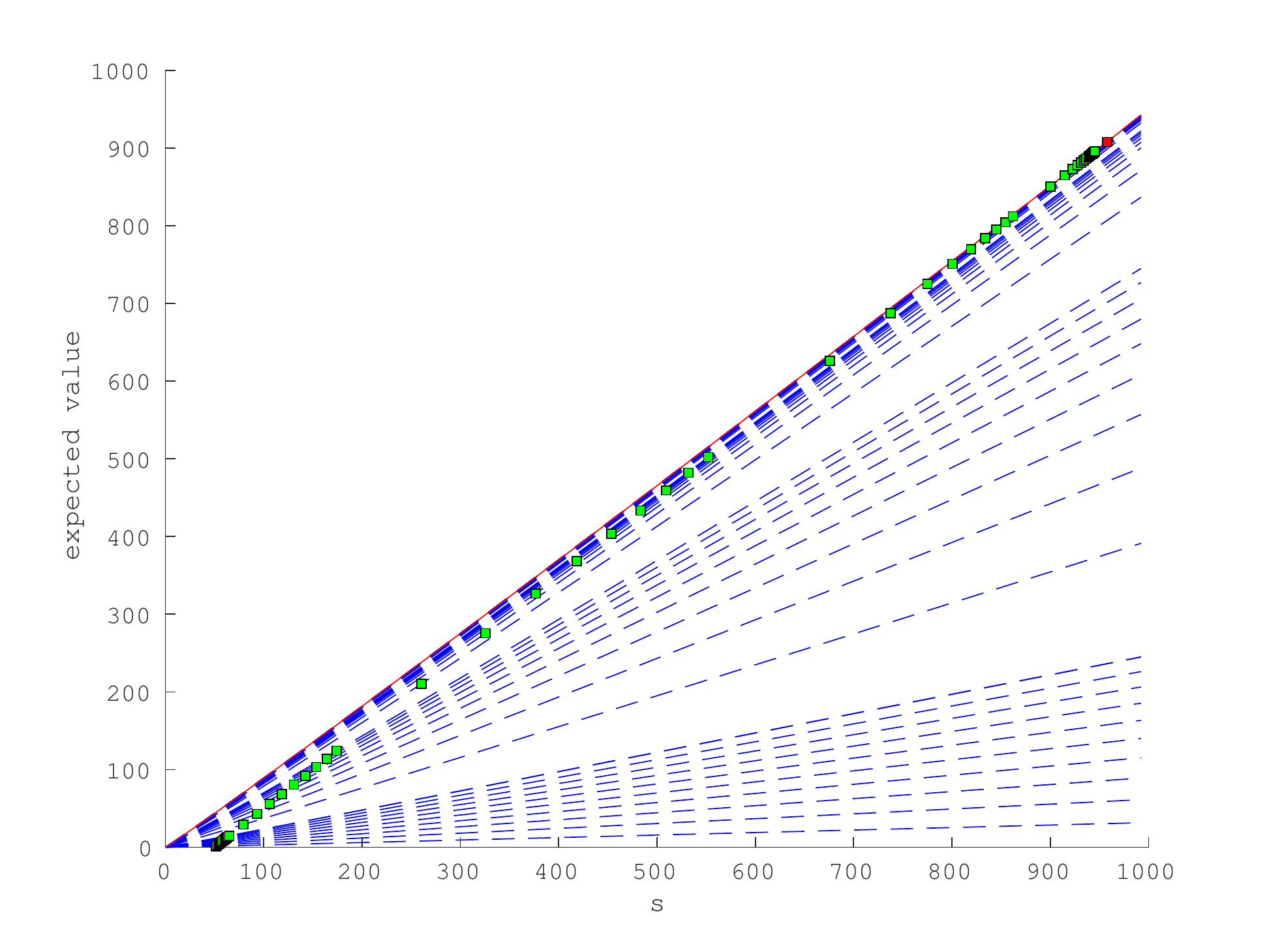}  \\
 spectrally negative case & spectrally positive case
 \end{tabular}
\end{minipage}
\caption{\footnotesize Plots of $V_c$ (dotted) for $\lambda=0.001$, $0.002$, $\ldots$, $0.009$, $0.01$, $0.02$, $\ldots$, $0.09$, $0.1$, $0.2$, $\ldots$, $0.9$, $1$, $2$, $\ldots$, $9$, $10$, $20$, $\ldots$, $190$, $200$ along with $V_{c, \infty}$ (solid).  The points at the optimal barriers are indicated by squares. 
} \label{figure_call_r}
\end{center}
\end{figure}


\begin{thebibliography}{99}	
\footnotesize
	\bibitem{AIZ} \sc Albrecher, H. and Ivanovs, J. and Zhou, X.
			\rm Exit identities for \lev processes observed at Poisson arrival times. {\it Bernoulli}, {\bf 22.3}, 1362--1382, (2014).
			
				\bibitem{AK} \sc Alili, L. and Kyprianou, A.E.
			\rm Some remarks on first passage of L\'evy processes, the American put and pasting principles. {\it Ann. Appl.Probab.}, {\bf 15} (3), 2062-2080, (2005).

			
%
%
%
%
	\bibitem{ATW2014}\sc Avanzi, B., Tu, V., and Wong, B. \rm  On optimal periodic dividend strategies in the dual model with diffusion. {\it Insur. Math. Econ.} {\bf 55}, 210-224, (2014).
%
\bibitem{APP}\sc Avram, F., Palmowski, Z., and Pistorius, M.R. \rm  On the optimal dividend problem for a spectrally negative L\'evy process. {\it Ann. Appl.Probab.} {\bf 17}, 156-180, (2007).

	\bibitem{APY}\sc Avram, F., P\'erez, J.L., and Yamazaki, K. \rm  Spectrally negative {L}\'evy processes with Parisian reflection below and classical reflection above. {\it Stochastic Process. Appl.}, {\bf 128} (1), 255-290, (2018).

	\bibitem{DP} \sc Dixit, A.K. and Pindyck, R.S.  {\it Investment under uncertainty.} \rm Princeton university press, (1994).
	%
%
%
%
%
%
	\bibitem{Egami_Yamazaki_2010_2}\sc Egami, M. and Yamazaki, K. \rm Phase-type fitting of scale functions for spectrally negative {L}\'evy processes.  {\it J. Comput. Appl. Math.} {\bf 264}, 1--22, (2014).
%
%
%
%

%
%
%

	\bibitem{Kou}\sc  Kou, S.G. \rm A jump-diffusion model for option pricing. {\it Manage. Sci.}, {\bf 48} (8), 1086--1101, (2002).
%
	\bibitem{KKR}\sc  Kuznetsov, A., Kyprianou, A.E., and Rivero, V. \rm The
	theory of scale functions for spectrally negative L\'evy processes. {\it L\'evy Matters II, Springer Lecture Notes in Mathematics}, (2013).
\bibitem{M}
\sc Mordecki, E.  \rm Optimal stopping and perpetual options
for L\'evy processes. {\it Finance Stoch.} {\bf 6}, 473--493, (2002).
%
%
%
%
%
%
%
%
	\bibitem{K} \sc Kyprianou, A.E. {\it Fluctuations of L\'evy processes with applications.} \rm Second edition, Springer, 
	Berlin, (2006).
	
		\bibitem{NPYY2017}\sc Noba, K., P\'erez, J.L., Yamazaki, K.  and Yano, K. \rm On optimal periodic dividend strategies for L\'evy risk processes.  {\it arXiv}, {\bf 1708.01678}, (2017).
	 
	
	
%
%
%
%
%
%
%
%
%
%
%
%
%
	\bibitem{YP2016}\sc P\'erez, J.L.\ and Yamazaki, K. \rm On the optimality of periodic barrier strategies for a spectrally positive L\'evy process.  {\it Insur. Math. Econ.}, {\bf 77}, 1-13, (2017).
	
				\bibitem{PYM}\sc P\'erez, J.L.\ and Yamazaki, K. \rm Mixed periodic-classical barrier strategies for L\'evy risk processes.  {\it arXiv}, {\bf 1609.01671}, (2016).
				
\bibitem{PS} \sc Peskir, G. and Shiryaev, A. {\it Optimal stopping and free-boundary problems.} \rm Birkh\"auser, (2006).
%
%
%
%
%
%
%
%
%
%
%
%
%
	
	
	
\end{thebibliography}
\end{document}